\documentclass[3p,letterpaper,twoside]{elsarticle}
\usepackage{amsmath,amssymb,graphicx,color,subcaption,epstopdf,amsthm}

\journal{Applied and Computational Harmonic Analysis}

\def\RR{{\mathbb{R}}}

\def\TT{{\mathbb{T}}}
\def\C{{\mathbb{C}}}

\def\R{{\mathbb{R}^d}}
\def\Z{{\mathbb{Z}^d}}

\def\Zplus{{\mathbb{Z}^{+}}}
\def\N{{\mathbb{N}^d}}
\def\NN{{\mathbb{N}}}

\def\bump{{\cC^{\infty}_{\rm c}(\R)}}

\def\Hd{{H\"older's}}

\def\d{{\rm d}}
\def\e{{\rm e}}
\def\j{{\rm j}}

\def\dual{{\varphi_{\rm d}}}

\def\bx{{\boldsymbol{x}}}
\def\by{{\boldsymbol{y}}}
\def\bu{{\boldsymbol{u}}}
\def\bomega{{\boldsymbol{\omega}}}
\def\bk{{\boldsymbol{k}}}
\def\bl{{\boldsymbol{\ell}}}

\def\zv{{\boldsymbol{0}}}

\def\cS{{\mathcal{S}}}
\def\cC{{\mathcal{C}}}

\def\cF{{\mathcal{F}}}

\def\f{{f}} 
\def\J{{J}} 
\def\L{{L}}
\def\h{{h}}
\def\P{{P_{\varphi}}}
\def\Ph{{P_{\varphi,\h}}}
\def\I{{I_{\varphi}}}
\def\Jh{{J_{\h}}}
\def\Ih{{I_{\varphi,\h}}}
\def\fr{D^r f}

\newcommand{\defeq}{:=}

\newcommand{\Frac}[2]{{#1}/{#2} }
\newcommand{\ip}[2]{\left\langle {#1},{#2}\right\rangle}
\newcommand{\w}[2]{\left\langle {#1}\right\rangle^{#2}}
\newcommand{\norm}[1]{\left\|{#1}\right\|}
\newcommand{\wn}[1]{\left\|{#1}\right\|_{L_{p,-\alpha}(\R)}} 

\DeclareMathOperator{\supp}{supp}
\DeclareMathOperator{\ess}{ess}

\newtheorem{thm}{Theorem}
\newtheorem{prop}{Proposition}

\newtheorem{lem}{Lemma}

\makeatletter
\def\@author#1{\g@addto@macro\elsauthors{\normalsize%
		\def\baselinestretch{1}%
		\upshape\authorsep#1\unskip\textsuperscript{%
			\ifx\@fnmark\@empty\else\unskip\sep\@fnmark\let\sep=,\fi
			\ifx\@corref\@empty\else\unskip\sep\@corref\let\sep=,\fi
		}%
		\def\authorsep{\unskip,\space}%
		\global\let\@fnmark\@empty
		\global\let\@corref\@empty  
		\global\let\sep\@empty}%
	\@eadauthor={#1}
}
\makeatother

\begin{document}

\begin{frontmatter}
	
	\title{Approximation of Non-Decaying Signals From Shift-Invariant Subspaces\tnoteref{t1}}
	\tnotetext[t1]{This research was funded by the European Research Council under
		the European Union's Seventh Framework Programme (FP7/2007-2013) / ERC Grant Agreement No. 267439 and the Swiss National Science Foundation under Grant 200020-162343.}
	
	\author{Ha Q. Nguyen\corref{corr}}
	\cortext[corr]{Corresponding author}
	\ead{ha.nguyen@epfl.ch}
	
	\author{Michael Unser}
	\ead{michael.unser@epfl.ch}
	
	\address{Biomedical Imaging Group, \'{E}cole Polytechnique F\'{e}d\'{e}rale de Lausanne (EPFL), Station 17, CH-1015, Lausanne, Switzerland}
	
	\begin{abstract}
		In our recent work, the sampling and reconstruction of non-decaying signals, modeled as members of weighted-$L_p$ spaces, were shown to be stable with an appropriate choice of the generating kernel for the shift-invariant reconstruction space. 
		In this paper, we extend the Strang-Fix theory to show that, for $d$-dimensional signals whose derivatives up to order $L$ are all in some weighted-$L_p$ space, the weighted norm of the approximation error can be made to go down as $O(h^L)$ when the sampling step $h$ tends to $0$. The sufficient condition for this decay rate is that the generating kernel belongs to a particular hybrid-norm space and satisfies the Strang-Fix conditions of order $L$. We show that the $O(h^L)$ behavior of the error is attainable for both approximation schemes using projection (when the signal is prefiltered with the dual kernel) and interpolation (when a prefilter is unavailable). The requirement on the signal for the interpolation method, however, is slightly more stringent than that of the projection because we need to increase the smoothness of the signal by a margin of $d/p+\varepsilon$, for arbitrary $\varepsilon >0$. This extra amount of derivatives is used to make sure that the direct sampling is stable. 
	\end{abstract}
	
	\begin{keyword}
		approximation theory\sep
		Strang-Fix conditions\sep
		shift-invariant spaces\sep
		spline interpolation\sep
		weighted $L_p$ spaces\sep
		weighted Sobolev spaces\sep
		hybrid-norm spaces
	\end{keyword}
	
\end{frontmatter}

\section{Introduction}
Sampling and reconstruction are important in signal processing because they provide an insightful connection between analog signals and their discrete representations. In the sampling procedure, oftentimes, a continuous-domain signal $f:\R\mapsto \C$ is uniformly sampled (with or without a \emph{prefilter}) at multi-integer multiples of some sampling step $h$ to produce a discrete-domain signal $c:\Z\mapsto \C$. The reconstruction, on the other hand, is commonly done by interpolating the samples $\{c[\bk]\}_{\bk\in\Z}$ with scaled and shifted copies of some kernel (generating function) $\varphi$ positioned on the grid $h\Z$. Precisely, the reconstructed signal  takes the (integer) shift-invariant form
\begin{align}\label{eq:interp}
	\tilde{f}(\bx) = \sum_{\bk\in\Z}c[\bk]\varphi\left(\frac{\bx}{\h}-\bk\right).
\end{align}
This interpolation model has been extensively used in the theory of splines~\cite{Schoenberg:1946,Schoenberg:1973,deBoor:1978,Schumaker:1981}. It is general enough to include
the celebrated reconstruction formula in Shannon's sampling theorem~\cite{Shannon:1949} in which the kernel $\varphi$ is replaced with the sinc function. Although the sinc-based interpolation guarantees exact recovery of bandlimited signals (or signals prefiltered with an ideal lowpass filter) whenever $1/h$ exceeds Nyquist's rate, the slow decay of sinc$(\bx)$ unfortunately prevents the application of this method in practice~\cite{Unser:2000}. For other choices of $\varphi$ with better localization properties, such as splines, exact reconstruction is no longer achievable but the quality of the approximation of a signal $f$ by such $\tilde{f}$ given in~\eqref{eq:interp} can be characterized as a power of the sampling step $h$ via the Strang-Fix theory. Specifically, in early 1970's, Strang and Fix~\cite{StrangF:1971} extended Schoenberg's work~\cite{Schoenberg:1946} and introduced the concept of \emph{controlled} approximation in which the $\ell_2$-norm of the sampled coefficients is bounded by the $L_2$-norm of the original signal. They showed that, for compactly supported $\varphi$, the error of the controlled approximation is bound as
\begin{align}\label{eq:S-F}
	\forall f\in H_2^{L}(\R),\quad\min_{c}\norm{f-\tilde{f}}_{L_2{(\R)}}\leq C_{\varphi,L}\cdot h^L\cdot\norm{f^{(L)}}_{L_2(\R)}, \quad\text{as } h\rightarrow 0,
\end{align}
\emph{if and only if}  $\varphi$ satisfies the Strang-Fix conditions of order $L$ so that the representation~\eqref{eq:interp} is able to reproduce all polynomials of degree less than $L$; this notion will be clarified later in Section~\ref{sec:Strang-Fix}. Here, $f^{(L)}$ is the $L$th derivative\footnote{To be precise, when $f$ is multivariate, $f^{(L)}$ is the summation of (the moduli of) all partial derivatives of order $L$ of $f$.} of $f$ and $H^L_2(\R)$ is the Sobolev space of $L_2$ functions whose first $L$ derivatives are all in $L_2(\R)$. The order $L$ in~\eqref{eq:S-F} is referred to in the literature as the order (power) of approximation. 

The original result of Strang and Fix has been extended in various directions,  including controlled $L_p$-approximation with globally supported (multi-) kernel~\cite{deBoorJ:1985,LightC:1992,JiaL:1993,Lei:1994}, uncontrolled $L_2$-approximation~\cite{deBoorDR:1994}, and finer estimations of the $L_2$-approximation error~\cite{Unser:1996,UnserD:1997,BluU:1999a,BluU:1999b,BluU:1999c}; interested readers are also referred to the surveys~\cite{Chui:1988,deBoor:1990,Light:1991}. More recently, the Strang-Fix theory was linked to the sampling of signals with finite rate of innovation~\cite{DragottiVB:2007}. Despite a rich literature on the Strang-Fix conditions, none of the existing results allows us to deal with the approximation of non-decaying (non-$L_p$) signals, such as sample paths of a Brownian motion, which can even grow at infinity. This is an important part that seems to be missing in the theory of \emph{sparse stochastic processes} recently developed by Unser \emph{et al.}~\cite{UnserT:2014,UnserTS:2014,UnserTAK:2014}. 

In this paper, a follow-up of our recent works on the sampling theory for non-decaying signals~\cite{NguyenU:2015b,NguyenU:2015a,NguyenU:2015c}, we provide an approximation theory for such objects. Recall that we showed~in~\cite{NguyenU:2015b} that both the sampling and reconstruction of weighted-$L_p$ signals, at a fixed sampling step, are stable, provided the generating kernel $\varphi$ lies in an appropriate  \emph{hybrid-norm space}, a concept closely related to the Wiener amalgams that are frequently used in time-frequency analysis~\cite{Wiener:1933,Feichtinger:1991a,Heil:2002}. Note that, in the \emph{direct} sampling scheme, where a prefilter is absent, not only the signal is required to live in a weighted-$L_p$ space, but also its first $d/p+\varepsilon$ derivatives, for some $\varepsilon >0$. In the spirit of~\cite{NguyenU:2015b}, we model non-decaying signals in this paper as members of the weighted space $L_{p,-\alpha}(\R)$ associated with the \emph{Sobolev weight} $(1+\|\cdot\|^2)^{-\alpha/2}$, where $\alpha\geq 0$ specifies the order of growth of the signals. In particular, $f\in L_{p,-\alpha}(\R)$ if $(1+\|\cdot\|^2)^{-\alpha/2} f\in L_p(\R)$.
We then extend the classical Strang-Fix theory to the approximation of such signals
for the  two common types of shift-invariant reconstructions: projection versus (direct) interpolation. 

In the \emph{projection} scheme, which provides the optimal $L_2$-approximation, the original signal is prefiltered with the \emph{dual} kernel $h^{-d}\dual\left(-\frac{\cdot}{h}\right)$~\cite{UnserD:1997} and the coefficients $\{c[\bk]\}_{\bk\in\Z}$ in~\eqref{eq:interp} are obtained by sampling the resulting signal with step size $h$. It means that the reconstructed signal is given by
\begin{align*} 
\tilde{f}_{\rm proj}(\bx) = \frac{1}{h^d}\sum_{\bk\in\Z}\ip{f}{\dual\left(\frac{\cdot}{h}-\bk\right)}\varphi\left(\frac{\bx}{\h}-\bk\right).
\end{align*} 
For this type of reconstruction, we show, in the first half of the paper, that if $\varphi$ belongs to an appropriate hybrid-norm space and at the same time satisfies the Strang-Fix conditions of order $L$, then the weighted-$L_p$ norm of the projection error is bounded as
\begin{align}\label{eq:weighted_proj}
\forall f\in H_{p,-\alpha}^{L}(\R),\quad\norm{f-\tilde{f}_{\rm proj}}_{L_{p,-\alpha}{(\R)}}\leq C_{\varphi,L,\alpha}\cdot h^L\cdot\norm{f^{(L)}}_{L_{p,-\alpha}(\R)}, \quad \text{as }h\rightarrow 0,
\end{align}
where the weighted Sobolev space $H^L_{p,-\alpha}(\R)$ is a collection of functions whose derivatives up to order $L$ are all in $L_{p,-\alpha}(\R)$. We want to remark that this result is the weighted version of~\cite[Theorem 2.2]{Lei:1994}.

In the \emph{interpolation} scheme, the coefficients are sampled directly from the original signal; hence the reconstructed signal takes the form
\begin{align*} 
\tilde{f}_{\rm int}(\bx) = \sum_{\bk\in\Z}f(h\bk)\,\varphi_{\rm int}\left(\frac{\bx}{\h}-\bk\right),
\end{align*} 
where $\varphi_{\rm int}$ is the \emph{interpolant} generated from the kernel $\varphi$~\cite{Unser:2000}. Similar to the projection case, we establish, in the second half of the paper, that if $\varphi$ is an element of a particular hybrid-norm space that satisfies the Strang-Fix condition of order $L$, then, given $r>d/p$, 
\begin{align}\label{eq:weighted_interp}
\forall f:\fr\in H_{p,-\alpha}^{L}(\R),\quad\norm{f-\tilde{f}_{\rm int}}_{L_{p,-\alpha}{(\R)}}\leq C_{\varphi,L,\alpha}\cdot h^L\cdot\norm{(\fr)^{(L)}}_{L_{p,-\alpha}(\R)}, \quad \text{as }h\rightarrow 0.
\end{align}
Here, $\fr$ is a combination of all \emph{fractional} derivatives up to order $r$ of $f$ defined in the frequency domain as $\fr\defeq\cF^{-1}\left\{(1+\|\cdot\|^2)^{\Frac{r}{2}}\cF f\right\}$ with $\cF$ being the Fourier transform operator. Informally speaking, the interpolation error can also be made to decay like $O(h^L)$, when $h$ tends to $0$, for functions whose derivatives up to order $L+d/p+\varepsilon$ live in some weighted-$L_p$ space, for arbitrary $\varepsilon>0$. This is not surprising because we need $d/p+\varepsilon$ derivatives to take care of the sampling, as indicated in~\cite{NguyenU:2015b}, and $L$ derivatives more to reach the target approximation order. To the best of our knowledge, the bound~\eqref{eq:weighted_interp} is new even in the unweighted $L_p$ case (when all instances of the subscript $\alpha$ disappear), although similar results exist for the direct interpolation in $L_2$~\cite{UnserD:1997} and $L_\infty$~\cite{LightC:1992}. The (unweighted) $L_p$ result presented in~\cite[Theorem 4.1]{JiaL:1993}, although similar to~\eqref{eq:weighted_interp}, does not fall into the realm of direct interpolation because the samples are taken from a smoothed version of the original signal. 

One of the challenges for the approximation in weighted spaces is that the beautiful Fourier-based methods commonly used in the Strang-Fix theory~\cite{StrangF:1971,deBoorDR:1994,Unser:1996,UnserD:1997,BluU:1999a,BluU:1999b,BluU:1999c} are no longer applicable,  even in the weighted-$L_2$ case, due to the lack of a Parseval-type relation. 
In proving the bounds~\eqref{eq:weighted_proj} and~\eqref{eq:weighted_interp}, we adapt the $L_p$-approximation techniques in~\cite{JiaL:1993,Lei:1994}, which are carried entirely in the space domain, but our analysis is much more involved because of the handling of the weights. We also heavily rely on the preliminary results in~\cite{NguyenU:2015b}. Other works that are closely related to the present paper are~\cite{Tomita:2006,AimarBM:2003} in which similar bounds were derived in the weighted-$L_p$ spaces associated with the so-called \emph{Muckenhoupt weights}~\cite{Muckenhoupt:1972}. These weights, however, are strikingly different from the Sobolev weights used in this paper. They are characterized by the boundedness of the Hardy-Littlewood maximal operator~\cite{HardyL:1930,Stein:1955,FeffermanS:1971} with respect to the weighted norm. Typical examples of the Muckenhoupt weights are $\|\cdot\|^{\alpha}$, for $\alpha$ being restricted in the interval $\left(-\Frac{d}{p},d-\Frac{d}{p}\right)$ (cf.~\cite{Kurtz:1980}). By contrast, the Sobolev weights $(1+\|\cdot\|^2)^{\alpha/2}$ can take arbitrary order $\alpha\in\RR$ and therefore give us more freedom in quantifying the growth or decay of the signals. Moreover, the Muckenhoupt weights are not well-suited to time-frequency analysis because they are generally not \emph{submultiplicative}, an important property that is satisfied by the Sobolev weights (cf.~\cite[Section 9]{Grochenig:2006}). 

The remainder of the paper is organized as follows: preliminary notions are introduced in Section~\ref{sec:prelim}; approximation error bounds for the projection and interpolation paradigms are derived in Sections~\ref{sec:proj} and~\ref{sec:interp}, respectively; proofs of several auxiliary results are given in Section~\ref{sec:proofs}.

\section{Preliminaries}	\label{sec:prelim}
\subsection{Notation}
All functions in this paper are mappings from $\R$ to $\C$ for a fixed dimension $d\geq 1$. Vectors in $\R$ are denoted by bold letters and their Euclidean norms are denoted by $\|\cdot\|$. The constants throughout the paper are denoted by $C$ with subscripts indicating the dependence of the constants on some parameters; we use the same notation for different constants that depend on the same set of parameters.
The restriction of a function $f$ on the multi-integer grid $\Z$ is denoted by $f[\cdot]$. $\NN$ is the set of natural numbers starting from zero and  $\Zplus$ is the set of positive integers, i.e., $\Zplus = \NN\setminus \{0\}$. For brevity, we denote by $\langle\cdot\rangle$ the Sobolev weighting function $(1+\|\cdot\|^2)^{1/2}$. 
For $1\leq p\leq \infty$, we use $p'$ to denote the H\"{o}lder conjugate of $p$ that satisfies $\frac{1}{p}+\frac{1}{p'}=1$ . 

$\bump$ is the space of smooth and compactly supported functions, $\cS(\R)$ is Schwartz' class of smooth and rapidly decaying functions, and $\cS'(\R)$ is the space of tempered distributions, which are continuous linear functionals on $\cS(\R)$. As usual, the notation $\ip{\cdot}{\cdot}$ is used interchangeably for the scalar product and for the action of a distribution on a test function.
The (distributional) Fourier transform $\hat{f}=\cF f$ of  a tempered distribution $f\in \cS'(\R)$ is also a tempered distribution defined as
\begin{align*}
\ip{\cF f}{\varphi}\defeq\ip{\hat{f}}{\varphi} \defeq\ip{f}{\hat{\varphi}}, \quad \text{for }\varphi\in \cS(\R),
\end{align*}
where 
\begin{align*}
\hat{\varphi}(\bomega) \defeq \int_{\R}\varphi(\bx) {\rm e}^{- {\rm j} \ip{\bomega}{\bx}}{\rm d}\bx.
\end{align*}
We denote the inverse Fourier-transform operator by $\cF^{-1}$. For a multi-index $\bl\in\N$, $|\bl|\defeq\sum_{i=1}^d \ell_i$ and $\partial^{\bl}$ is a shorthand for $\left(\Frac{\partial}{\partial x_1}\right)^{\ell_1}\cdots\left(\Frac{\partial}{\partial x_d}\right)^{\ell_d}$. The (distributional) partial derivative with respect to $\bl$ of a tempered distribution $f\in\cS'(\R)$ is also a tempered distribution defined as
\begin{align*}
	\ip{\partial^{\bl}f}{\varphi} &\defeq (-1)^{|\bl|}\ip{f}{\partial^{\bl}\varphi},\quad \text{for } \varphi\in\cS(\R).
\end{align*}
We also use the notation
\begin{align*}	
f^{(n)} &\defeq \sum_{\bl\in\N:|\bl|=n} |\partial^{\bl}f|.
\end{align*}
$\nabla\defeq \left(\Frac{\partial}{\partial x_1},\ldots,\Frac{\partial}{\partial x_d}\right)$ is  the gradient operator and $D_{\bu}  \defeq \ip{\nabla }{\bu}$ is the directional derivative operator with respect to  $\bu\in\R$. The shift and difference operators are defined as
$S_{\bu}f \defeq f(\cdot - \bu)$ and $\Delta_{\bu}f \defeq f - S_{\bu}f$, respectively. 
For $h>0$, $\sigma_h$ denotes the scaling operator given by $\sigma_{h}f \defeq f(\cdot/h)$.  
\subsection{Weighted Normed Spaces}
The spaces $L_p(\R)$ and $\ell_p(\Z)$ and their corresponding norms $\norm{\cdot}_{L_p(\R)}$ and $\norm{\cdot}_{\ell_p(\Z)}$ are defined as usual. We also need the hybrid-norm space $W_{p}(\R)$ which comprises all functions $f$ whose hybrid (mixed) norm
\begin{align*}
\norm{f}_{W_{p}(\R)}\defeq
\begin{cases}
\left(\int_{[0,1]^d}\left(\sum_{\bk\in\Z}|f(\bx+\bk)|\right)^p\d\bx\right)^{\Frac{1}{p}} &  1\leq p<\infty\\
\ess\sup_{\bx\in[0,1]^d}\sum_{\bk\in\Z}|f(\bx+\bk)| &  p=\infty
\end{cases}
\end{align*}
is finite. For any weighting function $w$, the weighted spaces $L_{p,w}(\R)$, $
\ell_{p,w}(\Z)$ and $W_{p,w}(\R)$ are defined with respect to the following weighted norms:
\begin{align*}
\norm{f}_{L_{p,w}(\R)}&\defeq \norm{f\cdot w}_{L_p(\R)},\\
\norm{c}_{\ell_{p,w}(\Z)}&\defeq \norm{c\cdot w[\cdot]}_{\ell_p(\Z)},\\
\norm{f}_{W_{p,w}(\R)}&\defeq \norm{f\cdot w}_{W_p(\R)}.
\end{align*}
When $w=\w{\cdot}{\alpha}$, for some $\alpha\in\RR$, we write $L_{p,\alpha}(\R)$ for $L_{p,w}(\R)$, $\ell_{p,\alpha}(\Z)$ for $\ell_{p,w}(\Z)$, and $W_{p,\alpha}(\R)$ for $W_{p,w}(\R)$. Note that, for $\alpha\geq 0$, the weight $w=\w{\cdot}{\alpha}$ is (weakly) submultiplicative, i.e.,
\begin{align*} 
\w{\bx+\by}{\alpha}\leq C_\alpha\, \w{\bx}{\alpha}\w{\by}{\alpha},\quad \forall \bx,\by\in\R,
\end{align*}
which is equivalent to
\begin{align*} 
\w{\bx+\by}{-\alpha}\leq C_\alpha\, \w{\bx}{\alpha}\w{\by}{-\alpha},\quad \forall \bx,\by\in\R.
\end{align*}
Furthermore, the weight $w=\w{\cdot}{\alpha}$ satisfies the Gelfand-Raikov-Shilov condition~\cite{GelfandRS:1964} that
\begin{align*} 
\lim_{n\rightarrow\infty} w(n\bx)^{\Frac{1}{n}} = 1,\quad \forall\,\bx\in\R.
\end{align*}
These two properties of $\w{\cdot}{\alpha}$, with $\alpha\geq 0$, will be crucial for us to manipulate weights.

Finally, let us define the weighted Sobolev spaces of integer and fractional orders. Given $1\leq p\leq \infty$ and $\alpha\in\RR$, the space $H^k_{p,\alpha}(\R)$ with $k\in\NN$ consists of all $f\in\cS'(\R)$ such that 
\begin{align*}	
\norm{f}_{H^k_{p,\alpha}(\R)} &\defeq \sum_{\bl\in\N:|\bl|\leq k} \norm{\partial^{\bl}f}_{L_{p,\alpha}(\R)}<\infty.
\end{align*}
It is straightforward that if $f\in H_{p,\alpha}^k(\R)$ then $f^{(n)}\in L_{p,\alpha}(\R)$, for all $n\leq k$. Meanwhile, the space $L^s_{p,\alpha}$ with $s\in \RR$ consists of all $f\in\cS'(\R)$ such that
\begin{align*}	
\norm{f}_{L^s_{p,\alpha}(\R)} &\defeq \norm{\cF^{-1}\left\{\w{\cdot}{s}\hat{f}\right\}}_{L_{p,\alpha}(\R)}<\infty.
\end{align*}
From here on, the term $\cF^{-1}\left\{\w{\cdot}{s}\hat{f}\right\}$  will be abbreviated as $D^sf$. When $s>0$, $D^{-s}$ is the Bessel potential of order $s$~\cite{Grafakos:2008b}.
We also need the hybrid weighted Sobolev space $H^{k,s}_{p,\alpha}$ which encompasses all $f\in\cS'(\R)$ such that $D^sf\in H^{k}_{p,\alpha}$. Note that, in the unweighted case ($\alpha=0$), it is not difficult to show that $H^{k,s}_{p}(\R)=L^{k+s}_p(\R)$, for $1<p<\infty$, using the Mikhlin-H\"{o}rmander theorem on Fourier multipliers (cf.~\cite[Chapter 5]{Grafakos:2008a} and~\cite[Chapter 6]{Grafakos:2008b}). For $\alpha\neq 0$, however, $H^{k,s}_{p,\alpha}(\R)$ is not necessarily the same as $L^{k+s}_{p,\alpha}(\R)$. This is due to the lack of a theory on weighted Fourier multipliers for the Sobolev weights; most of the existing literature are concerned with the  Muckenhoupt weights, instead~\cite{Kurtz:1980,EdmundsKM:2002,Krol:2014}.  
\subsection{Shift-Invariant Spaces of Non-Decaying Functions}
We are interested in the approximation of a non-decaying function living in the ambient space $L_{p,-\alpha}(\R)$, for some $\alpha \geq 0$, by an element in the (weighted) shift-invariant space $V_{p,-\alpha,\h}(\varphi)$ generated by some kernel $\varphi$ defined as
\begin{align*}
V_{p,-\alpha,h}(\varphi) &\defeq \left\{f=\sum_{\bk\in\Z}c[\bk]\varphi\left(\frac{\cdot}{h}-\bk\right): c\in\ell_{p,-\alpha}(\Z)\right\},
\end{align*}
where $h>0$ is a varying scale (sampling step). We write $V_{p,-\alpha}(\varphi)$ for $V_{p,-\alpha,1}(\varphi)$, write $V_{p,h}(\varphi)$ for $V_{p,0,h}(\varphi)$, and write $V_{p}(\varphi)$ for $V_{p,0,1}(\varphi)$. In addition to including many types of signal reconstruction models covered in the literature~\cite{Unser:2000}, this general formulation allows us to deal with (polynomially) growing signals. Similar to the unweighted case, we want to make sure that the (unscaled) space $V_{p,-\alpha}(\varphi)$ is a closed subspace of $L_{p,-\alpha}(\R)$ and each of its member $f\in V_{p,-\alpha}(\varphi)$ has an unambiguous representation in terms of the coefficients $c[\bk]$. It turns out that, as shown in~\cite[Theorem 2]{NguyenU:2015b}, this wish list will be fulfilled if the generating kernel $\varphi$ satisfies the following admissibility conditions:
\begin{itemize}
	\item[(i)] $\{\varphi(\cdot-\bk)\}_{\bk\in\Z}$ is a Riesz basis for $V_2(\varphi)$ or, equivalently, the Fourier tranform of the autocorrelation sequence, $\hat{a}_{\varphi}(\bomega)\defeq\sum_{\bk\in\Z}\left|\hat{\varphi}(\bomega+2\pi\bk)\right|^2$, is bounded from below and above for almost all $\bomega\in\R$;
	\item[(ii)] $\varphi$ belongs to the weighted hybrid-norm space $W_{q,\alpha}(\R)$ with $q\defeq\max(p,p')$. 
\end{itemize}  
We want to emphasize that the above conditions, though mathematically cumbersome, are by no means restrictive since they are easily satisfied by all interpolation kernels used in practice, and in particular B-splines~\cite{Unser:2000}.

\subsection{Strang-Fix Conditions}\label{sec:Strang-Fix}
There are multiple forms of the Strang-Fix conditions; the equivalence between them was initially shown for compactly supported functions~\cite{StrangF:1971} but then extended to kernels with global supports~\cite{LightC:1992,JiaL:1993}. The most common form of the Strang-Fix conditions is characterized in the frequency domain: a kernel $\varphi$ is said to satisfy the Strang-Fix conditions of order $L$ if  its Fourier transform $\hat{\varphi}$ satisfies
\begin{itemize}
	\item[(i)] $\hat{\varphi}(\zv)\neq 0$ and
	\item[(ii)] $\partial^{\bl}\hat{\varphi}(2\pi\bk)= 0, \quad \forall |\bl|\leq L-1,\forall \bk\in\Z\setminus\{\zv\}$.
\end{itemize}
These conditions are equivalent to the existence of a \emph{quasi-interpolant} $\varphi_{\rm QI}$ of order $L$~\cite{deBoorF:1973,deBoor:1987,ChuiD:1990} in the shift-invariant subspace $V_2(\varphi)$. This quasi-interpolant exactly interpolates all polynomials of degree (strictly) less than $L$, i.e.
\begin{align}\label{eq:quasi}
	\sum_{\bk\in\Z}\bk^{\bl}\varphi_{\rm QI}(\bx-\bk) = \bx^{\bl},\quad\forall |\bl|\leq L-1, \forall \bx\in\R,
\end{align}    
where $\bx^\bl$ stands for $x_1^{\ell_1}\cdots x_d^{\ell_d}$. Therefore, the Strang-Fix conditions of order $L$ can also be described as the ability of the space $V_2(\varphi)$ to reproduce polynomials of degree less than $L$. It is important to note that, for a particular $\varphi$, there are multiple choices for the quasi-interpolant within the subspace $V_2(\varphi)$, one of which is the interpolant $\varphi_{\rm int}$ that satisfies not only~\eqref{eq:quasi} but also the interpolating property
\begin{align}\label{eq:interpolating}
	\varphi_{\rm int}(\bk) = \delta[\bk],\quad\forall \bk\in\Z,
\end{align}
where $\delta[\cdot]$ denote the discrete unit impulse; the construction of this interpolant will be discussed in Section~\ref{sec:interp}.

Most importantly, the Strang-Fix conditions of order $L$ are necessary and sufficient for the controlled $L_2$-approximation of order $L$ that for any $f\in H_2^L(\R)$, there exists $\tilde{f}=\sum_{\bk\in\Z}c[\bk]\varphi\left(\Frac{\cdot}{h}-\bk\right)$ in $V_2(\varphi)$ such that 
\begin{itemize}
	\item[(i)] $\norm{c}_{\ell_2(\Z)} \leq C \cdot\|f\|_{L_2(\R)}$ and
	\item[(ii)] $\norm{f-\tilde{f}}_{L_2(\R)} \leq C \cdot h^L\cdot \norm{f^{(L)}}_{L_2(\R)}$, as $h\rightarrow 0$,
\end{itemize}
where the constants $C$ are independent of $f$. Note that the controllability of the approximation is dictated by the first bound, whereas the order of the approximation is described by the second bound. This beautiful connection between the approximation of order $L$ and the ability of the representation space to reproduce polynomials of degree less than $L$ lies at the core of the Strang-Fix theory and its various extensions~\cite{JiaL:1993,deBoorDR:1994,BluU:1999a}. Finally, it is handy to keep in mind that the B-spline of order $L$~\cite{UnserAE:1993a,UnserAE:1993b} satisfies the Strang-Fix conditions of order $L$.
\section{Projection Error Bound}\label{sec:proj}
In this section, we derive the error bound for the approximation of a non-decaying function in the weighted Sobolev space $ H^L_{p,-\alpha}(\R)$ by its projection onto the shift-invariant space $V_{p,-\alpha,h}(\varphi)$.  
Assume throughout this section that the kernel $\varphi$ is such that $\left\{\varphi(\cdot-\bk)\right\}_{\bk\in\Z}$ is a Riesz basis for $V_2(\varphi)$. This condition guarantees~\cite{Unser:2000} that the dual kernel $\dual$ exists and is given in the Fourier domain by
\begin{align*}
	\widehat{\dual}(\bomega) = \frac{\hat{\varphi}(\bomega)}{\sum_{\bk\in\Z}|\hat{\varphi}(\bomega+2\pi\bk)|^2}. 
\end{align*}
Let us define the operator
\begin{align*}
P_{\varphi,h}: f\mapsto  \tilde{f}_{\rm proj}=\sum_{\bk\in\Z}c[\bk]\varphi\left(\frac{\cdot}{h}-\bk\right),
\end{align*}
where, for each $\bk\in\Z$, the coefficient $c[\bk]$ is given by 
\begin{align*}
	c[\bk] = \frac{1}{h^d}\int_{\R} f(\by)\dual\left(\frac{\by}{h}-\bk\right)\d \by.
\end{align*}
In the language of signal processing, $c[\bk]$ is the result of prefiltering the signal $f$ with the filter $h^{-d}\dual\left(-\frac{\cdot}{h}\right)$ followed by a sampling at location $h\bk$. 
We write $P_{\varphi}$ for $P_{\varphi,1}$. It is well known in the (unweighted) $L_2$ case that $P_{\varphi,h}$ is an orthogonal projector from $L_2(\R)$ onto the subspace $V_{2,h}(\varphi)$ and therefore provides the best $L_2$-approximation. In the weighted-$L_p$ setup, orthogonality no longer exists but the operator $P_{\varphi,h}$ still behaves properly. In particular, the following result shows that $P_{\varphi,h}$ is a bounded projector from $L_{p,-\alpha}(\R)$ onto $V_{p,-\alpha,h}(\varphi)$ whose norm is bounded as the scale $h$ tends to $0$. The essential condition for that to hold true is that the generating kernel $\varphi$ is a member of an appropriate weighted hybrid-norm space.

\begin{thm}\label{thm:projection}
	Let $1\leq p\leq \infty$ and $\alpha\geq 0$. If $\varphi\in W_{q,\alpha}(\R)$ with $q\defeq\max(p,p')$ and $\left\{\varphi(\cdot-\bk)\right\}_{\bk\in\Z}$ is a Riesz basis for $V_2(\varphi)$, then, for all $h>0$, $V_{p,-\alpha,h}(\varphi)$ is a closed subspace of $L_{p,-\alpha}(\R)$ and $\Ph$ is a projector from $L_{p,-\alpha}(\R)$ onto $V_{p,-\alpha,h}(\varphi)$. Furthermore, there exists a constant $C_{\varphi,\alpha}$ such that 
	\begin{align}\label{eq:projector_bound_h}
		\wn{\Ph f} \leq C_{\varphi,\alpha}\wn{f}, \quad\forall f\in L_{p,-\alpha}(\R), \forall h\in (0,1).
	\end{align}
\end{thm}
\begin{proof}
Since $\varphi\in W_{q,\alpha}(\R)$ and $\left\{\varphi(\cdot-\bk)\right\}_{\bk\in\Z}$ is a Riesz basis for $V_2(\varphi)$, it is known from~\cite[Theorems 1 \& 2]{NguyenU:2015b} that  $V_{p,-\alpha}(\varphi)$ is a closed subspace of $L_{p,-\alpha}(\R)$ and $\P$ is a bounded projector from $L_{p,-\alpha}(\R)$ onto $V_{p,-\alpha}(\varphi)$. We now divide the rest of the proof into several steps.

First, we show that  $V_{p,-\alpha,h}(\varphi)$ is a subspace of $L_{p,-\alpha}(\R)$, for all $h>0$. Given $f\in V_{p,-\alpha,h}(\varphi)$, it is clear that $\sigma_{1/h}f\in V_{p,-\alpha}(\varphi)\subset L_{p,-\alpha}(\R)$. On the other hand,  
\begin{align}
	\wn{f}^p  &= h^{d}\int_{\R} \w{\h \bx}{-\alpha p}\left|(\sigma_{1/h}f)(\bx)\right|^p\d\bx\nonumber\\
	&\leq h^d\cdot\max (1,h^{-\alpha p}) \int_{\R} \w{\bx}{-\alpha p}\left|(\sigma_{1/h}f)(\bx)\right|^p\d\bx\nonumber\\
	&= h^d\cdot\max (1,h^{-\alpha p}) \cdot	\wn{\sigma_{1/h}f}^p.\label{eq:norm_h}
\end{align}
This implies that $f$ also belongs to $L_{p,-\alpha}(\R)$, or $V_{p,-\alpha,h}(\varphi)$ is a subspace of $L_{p,-\alpha}(\R)$, for all $h>0$.

Second, we show that $V_{p,-\alpha,h}(\varphi)$ is closed under the norm of $L_{p,-\alpha}(\R)$, for all $h>0$. 
Let $\{f_n\}$ be a sequence in $V_{p,-\alpha,h}(\varphi)$ such that $f_n\rightarrow f$ in $L_{p,-\alpha}(\R)$ as $n\rightarrow \infty$. Similar to~\eqref{eq:norm_h}, we have that
\begin{align*}
	\wn{\sigma_{1/h}f_n - \sigma_{1/h}f} \leq h^{-d/p}\cdot \max(1,h^{\alpha}) \cdot\wn{f_n-f},
\end{align*}
which implies that $\sigma_{1/h}f_n\rightarrow \sigma_{1/h}f$ in $L_{p,-\alpha}(\R)$ as $n\rightarrow \infty$. As $\left\{\sigma_{1/h}f_n\right\}$ is a sequence in $V_{p,-\alpha}(\varphi)$, it follows from the closedness of $V_{p,-\alpha}(\varphi)$ that $\sigma_{1/h} f \in V_{p,-\alpha}(\varphi)$, or $f\in V_{p,-\alpha,h}(\varphi)$. This shows the closedness of $V_{p,-\alpha,h}(\varphi)$.

Third, we show that $\Ph$ is a projector that maps $L_{p,-\alpha}(\R)$ to $V_{p,-\alpha,h}(\varphi)$, for all $h>0$. Observe that $P_{\varphi,h} = \sigma_{h}P_{\varphi}\sigma_{1/h}$. From~\eqref{eq:norm_h}, $\sigma_{1/h}$ maps $L_{p,-\alpha}(\R)$ to itself. It is also known that $P_{\varphi}$ maps $L_{p,-\alpha}(\R)$ to  $V_{p,-\alpha}(\varphi)$ and $\sigma_{h}$ maps $V_{p,-\alpha}(\varphi)$ to $V_{p,-\alpha,h}(\varphi)$. Therefore, $\Ph$ maps $L_{p,-\alpha}(\R)$ to $V_{p,-\alpha,h}(\varphi)$. The idempotence of $\Ph$ can be easily verified as 
$$
P_{\varphi,h}^2 = \sigma_{h}P_{\varphi}\sigma_{1/h}\sigma_{h}P_{\varphi}\sigma_{1/h} = \sigma_{h}P^2_{\varphi}\sigma_{1/h} = \sigma_{h}P_{\varphi}\sigma_{1/h} = P_{\varphi,h},
$$
where we have relied on the idempotence of the projector $P_{\varphi}$.

Finally, we show the bound~\eqref{eq:projector_bound_h}. Let us consider the weighting function $w_{h}(\bx)\defeq \w{\h\bx}{\alpha}$. It is easy to see that $w_h$ satisfies 
\begin{align}\label{eq:submulti}
	w_h(\bx+\by)\leq C_{\alpha}\, w_h(\bx)w_h(\by),\quad \forall \bx,\by\in\R,\forall h>0.
\end{align}
By a change of variable and from the last bound in the proof of~\cite[Theorem 1]{NguyenU:2015b}, we have that, for all $h>0$,
\begin{align}
	\norm{P_{\varphi,h}f}_{L_{p,-\alpha}(\R)} 
	&= \norm{\sigma_{h}P_{\varphi}\sigma_{1/h}f}_{L_{p,-\alpha}(\R)}\nonumber\\
	&= h^{d/p}\cdot\norm{P_{\varphi}(\sigma_{1/h}f)}_{L_{p,\Frac{1}{w_h}}(\R)}\nonumber\\
	&\leq h^{d/p}\cdot C_{\alpha}^2\cdot\norm{\varphi}_{W_{p,w_h}(\R)}\norm{\dual}_{W_{p',w_h}(\R)}\norm{\sigma_{1/h}f}_{L_{p,\Frac{1}{w_h}}(\R)}\nonumber\\
	&= C_{\alpha}^2\cdot\norm{\varphi}_{W_{p,w_h}(\R)}\norm{\dual}_{W_{p',w_h}(\R)}\norm{f}_{L_{p,-\alpha}(\R)},\label{eq:projector_bound1}
\end{align}
where $C_{\alpha}$ is precisely the constant in~\eqref{eq:submulti} that does not depend on $h$. On the other hand, according to~\cite[Proposition 6]{NguyenU:2015b}, both $\varphi$ and $\dual$ are elements of $W_{q,\alpha}(\R)$. Since $q=\max(p,p')$, it must be that $\varphi\in W_{p,\alpha}(\R)$ and $\dual\in W_{p',\alpha}(\R)$. Moreover, the assumption that $h\in (0,1)$ gives
\begin{align}
	\norm{\varphi}_{W_{p,w_h}(\R)} \leq \norm{\varphi}_{W_{p,\alpha}(\R)} < \infty,\label{eq:projector_bound2}
\end{align}
and
\begin{align}
	\norm{\dual}_{W_{p',w_h}(\R)} \leq \norm{\dual}_{W_{p',\alpha}(\R)} < \infty.\label{eq:projector_bound3}
\end{align}
Putting together~\eqref{eq:projector_bound1}, \eqref{eq:projector_bound2}, and~\eqref{eq:projector_bound3} yields the desired bound~\eqref{eq:projector_bound_h}.
\end{proof}

The main result of this section is as follows:
\begin{thm}\label{thm:main}
	Let $1\leq p\leq \infty$, $L\in\Zplus$, and $\alpha\geq 0$. Assume that $\varphi\in W_{q,L+\alpha}(\R)$ with $q\defeq\max(p,p')$ and that $\left\{\varphi(\cdot-\bk)\right\}_{\bk\in\Z}$ is a Riesz basis for $V_2(\varphi)$. Assume also that $\varphi$ satisfies the Strang-Fix conditions of order $L$. Then, there exist a constant $C_{\varphi,L,\alpha}$ such that, for all $f\in H^{L}_{p,-\alpha}(\R)$,
	\begin{align}\label{eq:bound_error} 
	\wn{f-P_{\varphi,h}f}  \leq  C_{\varphi,L,\alpha}\cdot\h^L\cdot \wn{f^{(L)}},
	\end{align}
	when $h\rightarrow 0$.
\end{thm}

In what follows, we break the proof of Theorem~\ref{thm:main} into several small results.
Let us begin by defining the smoothing operator $\Jh$ as
\begin{align}\label{eq:smoothing}
\Jh:  f \mapsto \int_{\R}\left(f - \Delta_{\h\bu}^{\L}f\right)(\cdot) \chi\left(\bu\right)\d \bu,
\end{align}
with some underlying function $\chi\in\bump$ such that $\supp (\chi)\subset [-1,1]^d$ and $\int_{\R}\chi(\bu)\d\bu=1$.
This smoothing operator was also exploited in~\cite{JiaL:1993,Lei:1994}.

Expanding $\Delta_{\h\bu}^{\L}f$ as
\begin{align*}
\Delta_{\h\bu}^{\L} f = \sum_{n=0}^{\L} (-1)^{n}\binom{\L}{n}f(\cdot - n\h\bu),
\end{align*}
we obtain
\begin{align*}
f - \Delta_{\h\bu}^{\L} f = \sum_{n=1}^{\L} (-1)^{n-1}\binom{\L}{n}f(\cdot - n\h\bu).
\end{align*}
Therefore, $\Jh$ can also be expressed as
\begin{align*}
\Jh f &=  \sum_{n=1}^{\L} (-1)^{n-1}\binom{\L}{n} \int_{\R} f(\cdot - n\h\bu)\chi\left(\bu\right)\d \bu.
\end{align*}
This means that $\Jh$ is a convolution operator: $\Jh f= f * \psi_h$, where 
\begin{align}\label{eq:smoothingKernel}
\psi_h \defeq \sum_{n=1}^{\L} (-1)^{n-1}\binom{\L}{n}\frac{1}{(n\h)^d}\sigma_{n\h}\chi.
\end{align}

The following result shows that the weighted norm of the error between a function $f\in H^L_{p,-\alpha}(\R)$ and its smoothed version $\Jh f$ is $O(h^L)$ as $h$ tends to $0$.
\begin{prop}\label{thm:f-Jf}
	For $1\leq p\leq \infty$, $L\in \Zplus$, $\alpha\geq 0$, and $\Jh$ being the smoothing operator defined in~\eqref{eq:smoothing}, there exists a constant $C_{L,\alpha}$ such that, for all $f\in H^L_{p,-\alpha}(\R)$ and for all $h\in (0,1)$,
	\begin{align}\label{eq:bound_f-Jf}
	\wn{f-\Jh f} \leq C_{L,\alpha} \cdot\h^{L}\cdot\wn{f^{(L)}}.
	\end{align}
\end{prop}
\begin{proof}
We first need the following two lemmas whose proofs can be found in Section~\ref{sec:proofs}.
\begin{lem}\label{thm:Bspline}
	Let $L\in\NN$ and let $\beta^{L-1}$ be the (1-D) B-spline of order $(L-1)$ given by the $L$-fold convolution
	\begin{align*}
		\beta^{L-1}\defeq\underbrace{\beta^0 * \beta^0*\cdots * \beta^0}_{ L \text{ times}},
	\end{align*}
	where
	$$
		\beta^0(x) \defeq
		\begin{cases}
			1, & 0< x < 1\\
			0, & \text{otherwise}
		\end{cases}.
	$$
	Then, for all $f\in\cS'(\R)$, one has
	\begin{align}\label{eq:desired}
	\Delta^{L}_{\bu}f = \int_{\RR}D^{L}_{\bu}f(\cdot - t\bu) \beta^{L-1}(t)\d t.
	\end{align}
\end{lem}

\begin{lem}\label{thm:direction} 
	Let $L\in\NN$ and $\bu\in\R$. If $f\in\cS'(\R)$ such that its partial derivatives up to order $L$ are locally integrable functions, then
	\begin{align}\label{eq:direction}
	\left|D^L_{\bu}f(\bx)\right| \leq \|\bu\|^L_{\infty} \cdot f^{(L)}(\bx),\quad\forall \bx\in\R,
	\end{align}
	where $\|\bu\|_{\infty}\defeq \max\{|u_1|,\ldots,|u_d|\}$.
\end{lem}

We remark that Lemma~\ref{thm:Bspline} is an extension of Peano's theorem~\cite[page 70]{Davis:1975} for smooth functions. It is needed to avoid the density argument in the proof of~\cite[Theorem 3.3]{JiaL:1993} that is unavailable in the weighted case. Let us continue with the proof of Proposition~\ref{thm:projection}. Observe that
	\begin{align*}
		(f - \Jh f)(\bx) &= f(\bx)\int_{\R} \chi\left(\bu\right)\d \bu - \int_{\R}\left(f - \Delta_{\h\bu}^{\L}f\right)(\bx) \chi\left(\bu\right)\d \bu\\
		&= \int_{\R} \Delta_{\h\bu}^{\L}(\bx) \chi\left(\bu\right)\d\bu.
	\end{align*}	
	From Lemma~\ref{thm:Bspline} and by taking into account the fact that $\supp(\chi)\subset [-1,1]^d$ and $\supp(\beta^{L-1})=[0,L]$, we write
		\begin{align*}
		(f - \Jh f)(\bx)
		&= \int_{[-1,1]^d} \int_{0}^{L}D^{L}_{\h\bu}f(\bx - t\h\bu) \beta^{L-1}(t)\chi\left(\bu\right)\d \bu\,\d t.
		\end{align*}
	It then follows from Minkowski's inequality and Lemma~\ref{thm:direction} that
	\begin{align}
		\wn{f-\Jh f} &\leq \int_{[-1,1]^d} \int_{\RR}\wn{D^{L}_{\h\bu}f(\cdot - t\h\bu)} \beta^{L-1}(t)\chi\left(\bu\right)\d \bu\,\d t\nonumber\\
		&\leq \int_{[-1,1]^d} \int_{\RR}\|h\bu\|^L_{\infty}\cdot\wn{f^{(L)}(\cdot - t\h\bu)} \beta^{L-1}(t)\chi\left(\bu\right)\d \bu\,\d t\nonumber\\
		&\leq h^{L}\cdot\int_{[-1,1]^d} \int_{\RR}\wn{f^{(L)}(\cdot - t\h\bu)} \beta^{L-1}(t)\chi\left(\bu\right)\d \bu\,\d t.\label{eq:double_int}
	\end{align}
	On the other hand,
	\begin{align*}
		\wn{f^{(L)}(\cdot - t\h\bu)} &= \left(\int_{\R}\left|\w{\bx}{-\alpha}f^{(L)}(\bx - t\h\bu)\right|^p\d\bx\right)^{1/p}\\
		&\leq C_{\alpha} \w{th\bu}{\alpha}\left(\int_{\R}\left|\w{\bx-th\bu}{-\alpha}f^{(L)}(\bx - t\h\bu)\right|^p\d\bx\right)^{1/p}\\
		&= C_{\alpha} \w{th\bu}{\alpha} 	\wn{f^{(L)}}.
	\end{align*}
	Thus, for $t\in [0,L]$ and $h\in(0,1)$,
	\begin{align}\label{eq:norm_bound}
	\wn{f^{(L)}(\cdot - t\h\bu)}&\leq  C_{\alpha}\, L^{\alpha}\w{\bu}{\alpha} 	\wn{f^{(L)}}.
	\end{align}
	Combining~\eqref{eq:norm_bound} with~\eqref{eq:double_int} leads to 
	\begin{align*}
	\wn{f-\Jh f} 
	&\leq C_{\alpha}\,L^{\alpha}\cdot\h^L\cdot  \wn{f^{(L)}}\int_{0}^{L}\beta^{L-1}(t)\d t\int_{[-1,1]^d} \w{\bu}{\alpha}\chi\left(\bu\right)\d \bu\\
	&=C_{L,\alpha} \cdot\h^{L}\cdot\wn{f^{(L)}},
	\end{align*}
	which completes the proof.
\end{proof}


\begin{prop}\label{thm:error_Jf}
	Assume that $1\leq p\leq \infty$, $L\in\Zplus$, and $\alpha\geq 0$. Let $q\defeq\max(p,p')$ and let $\Jh$ be the smoothing operator defined in~\eqref{eq:smoothing}. If $\varphi$ is an element of $W_{q,L+\alpha}(\R)$   that satisfies the Strang-Fix conditions of order $L$,  then there exists a constant $C_{\varphi,L,\alpha}$ such that, for all $f\in H^{L}_{p,-\alpha}(\R)$ and for all $h\in (0,1)$,
	\begin{align*} 
	\left\|\Jh f-\Ph \Jh f\right\|_{L_{p,-\alpha}(\R)} \leq C_{\varphi,L,\alpha}\cdot h^L \cdot \wn{f^{(L)}}.
	\end{align*}
\end{prop}

\begin{proof}
	We begin the proof with a lemma; its proof is given in Section~\ref{sec:proofs}.
	\begin{lem}\label{thm:J_bounded}
		Let $w_h(\bx)\defeq \w{h\bx}{\alpha}$, $\alpha\geq 0$. Then, there exists a constant $C_{L,\alpha}$ such that, for all $f\in L_{p,-\alpha}(\R)$ and for all $h\in (0,1)$,
		\begin{align}\label{eq:bound_Jf}
		\left\|(\sigma_{\Frac{1}{h}}\Jh f)[\cdot]\right\|_{\ell_{p,1/w_h}(\Z)} \leq C_{L,\alpha}\cdot h^{-d/p}\cdot \left\|f\right\|_{L_{p,-\alpha}(\R)}.
		\end{align}
	\end{lem}
	
	Let us now put $g\defeq \Jh f$ and $e\defeq g-\Ph g$. It is clear that $g$ is infinitely differentiable. For $\bx\in\R$, let $R_{\bx}$ denote the remainder of the order-$(L-1)$ Taylor series  of function $g$ about $\bx$. 
	Since $\varphi$ satisfies the Strang-Fix conditions of order $L$, it is known~\cite{UnserD:1997} that $\Ph$ maps every polynomial of degree less than $L$ to itself. Therefore, it is possible to write
	\begin{align}
		e(\bx)	& = - \sum_{\bl\in\Z}c_{\bx}[\bl] \varphi\left(\frac{\bx}{h} - \bl\right),  \label{eq:error_g}
	\end{align} 
	where the sequence $c_{\bx}$ is given by
	\begin{align}\label{eq:coefficients}
		c_{\bx}[\bl] &\defeq \frac{1}{h^d}\int_{\R}R_{\bx}(\by)\dual\left(\frac{\by}{h} - \bl\right)\d\by.
	\end{align}
	 The weighted-$L_p$ norm of the projection error is then bounded as
	\begin{align}
	\left\|e\right\|^p_{L_{p,-\alpha}(\R)} &= \sum_{\bk\in\Z}\int_{[0,h]^d}  \left|\w{\bx+h\bk}{-\alpha}e(\bx+h\bk)\right|^p\d \bx\nonumber\\
	 &= h^{d}\cdot\int_{[0,1]^d}  \sum_{\bk\in\Z}\left|\w{\h\bx+h\bk}{-\alpha}e(\h\bx+h\bk)\right|^p\d \bx\nonumber\\
	&= h^{d}\cdot\int_{[0,1]^d} \sum_{\bk\in\Z} \left|\w{\h\bx+h\bk}{-\alpha}\sum_{\bl\in\Z}c_{\h\bx+h\bk}[\bl]\cdot \varphi\left(\bx+ \bk-\bl\right)\right|^p\d \bx\nonumber\\
	&\leq C_{\alpha}\cdot h^{d}\cdot\int_{[0,1]^d} \sum_{\bk\in\Z} \left(\sum_{\bl\in\Z} \w{h\bk}{-\alpha}\left|c_{\h\bx+h\bk}[\bk-\bl]\right|\cdot \left|\varphi\left(\bx+\bl\right)\right|\right)^p\d \bx\label{eq:submulti3}.
	\end{align}
	The last estimate is due to a change of variable and to the fact that $\w{h\bx+h\bk}{-\alpha}\leq C_\alpha \w{h\bk}{-\alpha}$, $\forall\bx\in[0,1]^d$, $\forall \h\in (0,1)$. Let us define the two sequences: $c_{\bx,\bl}[\bk]\defeq \w{h\bk}{-\alpha}|c_{\h\bx+h\bk}[\bk-\bl]|$ and $\varphi_{\bx}[\cdot] =|\varphi(\bx+\cdot)|$, for each $\bx\in [0,1]^d$ and each $\bl\in\Z$. Plugging these notations into~\eqref{eq:submulti3} and applying Minkowski's inequality, we obtain
		\begin{align}
		\left\|e\right\|^p_{L_{p,-\alpha}(\R)} 
		&\leq C_{\alpha}\cdot h^{d}\cdot\int_{[0,1]^d} \sum_{\bk\in\Z} \left(\sum_{\bl\in\Z}c_{\bx,\bl}[\bk]\cdot \varphi_{\bx}[\bl]\right)^p\d \bx\nonumber\\
		&\leq 	C_{\alpha}\cdot h^{d}\cdot\int_{[0,1]^d} \left(\sum_{\bl\in\Z} \|c_{\bx,\bl}\|_{\ell_{p}(\Z)}\cdot \varphi_{\bx}[\bl] \right)^p\d \bx.\label{eq:step1}
		\end{align}
		
	We now proceed to bound the quantity $\left\|c_{\bx,\bl}\right\|_{\ell_{p}(\Z)}$. By Taylor's theorem
	\begin{align}
		R_{\h\bx+\h\bk}(\h\by + \h\bk) &= \int_{0}^{1}\frac{(1-\tau)^{L-1}}{(L-1)!}S_{\tau\h\by+(1-\tau)\h\bx}D^{L}_{\h\by-\h\bx}(\Jh f)(\h\bk) \d\tau\nonumber\\
		        &= \int_{0}^{1}\Jh T_{\by,\tau} f(\h\bk) \d\tau,\label{eq:swap}
	\end{align}
	where the operator $T_{\by,\tau}$ is defined as
	\begin{align}\label{eq:operator}
	T_{\by,\tau} \defeq \frac{(1-\tau)^{L-1}}{(L-1)!} S_{\tau\h\by+(1-\tau)\h\bx}D^L_{\h\by-\h\bx}.
	\end{align}
	Note that the swapping of $T_{\by,\tau}$ and $\Jh$ in~\eqref{eq:swap} is justified because $\Jh$ is a convolution operator and hence commutes with differential and shift operators. From~\eqref{eq:coefficients} and the definition of $c_{\bx,\bl}$, one has
	\begin{align}
		c_{\bx,\bl}[\bk] &=\w{h\bk}{-\alpha}\int_{\R}R_{\h\bx+\h\bk}(\h\by + \h\bk)\dual\left(\by+\bl\right)\d\by\nonumber\\
		&=\int_{\R}\dual\left(\by+\bl\right)\int_{0}^{1}\frac{1}{w_h(\bk)}\cdot\Jh T_{\by,\tau} f(\h\bk)\d\tau\d \by,
	\end{align}
	where $w_h\defeq \w{\h\cdot}{\alpha}$. By Minkowski's inequality and by Lemma~\ref{thm:J_bounded}
	\begin{align}
		\|c_{\bx,\bl}\|_{\ell_{p}(\Z)} &\leq \int_{\R}\left|\dual\left(\by+\bl\right)\right|\int_{0}^{1} \left\|(\sigma_{1/h}\Jh T_{\by,\tau}f)[\cdot]\right\|_{\ell_{p,1/w_h}(\Z)}\d\tau \d\by\nonumber\\
		&\leq C_{L,\alpha} \cdot h^{-d/p}\int_{\R}\left|\dual\left(\by+\bl\right)\right|\int_{0}^{1} \left\| T_{\by,\tau}f\right\|_{L_{p,-\alpha}(\R)}\d\tau \d\by\label{eq:c_T}.
	\end{align}
	On the other hand
	\begin{align}
	 \left\| T_{\by,\tau}f\right\|_{L_{p,-\alpha}(\R)} &=  \frac{(1-\tau)^{L-1}}{(L-1)!}  \cdot \left\| S_{\tau\h\by+(1-\tau)\h\bx}D^L_{\h\by-\h\bx}f\right\|_{L_{p,-\alpha}(\R)} \nonumber\\
	&\leq C_{L}\cdot \left\| f^{(L)}(\cdot -\tau\h\by-(1-\tau)\h\bx)\right\|_{L_{p,-\alpha}(\R)}\|\h\by-\h\bx\|^{L}\label{eq:lemma2}\\
	&\leq C_{L,\alpha}\cdot\left\| f^{(L)}\right\|_{L_{p,-\alpha}(\R)}\cdot \w{\tau\h\by+(1-\tau)\h\bx}{\alpha}\|h\by-h\bx\|^{L}\label{eq:submulti4}\\
	&\leq C_{L,\alpha}\cdot h^L\cdot\left\| f^{(L)}\right\|_{L_{p,-\alpha}(\R)} \cdot\w{\by-\bx}{\alpha}\|\by-\bx\|^{L}\label{eq:para_bounded}\\
	&\leq C_{L,\alpha}\cdot h^L\cdot\left\| f^{(L)}\right\|_{L_{p,-\alpha}(\R)} \cdot\w{\by-\bx}{L+\alpha},\label{eq:T}
	\end{align}
	where~\eqref{eq:lemma2} follows from Lemma~\ref{thm:direction}; \eqref{eq:submulti4} is due to the submultiplicativity of the weight $\w{\cdot}{\alpha}$; and~\eqref{eq:para_bounded} is because $h,\tau\in(0,1)$ and $\bx\in [0,1]^d$. Putting~\eqref{eq:c_T} and~\eqref{eq:T} together
		\begin{align}
		\|c_{\bx,\bl}\|_{\ell_{p}(\Z)} 
		&\leq C_{L,\alpha}\cdot h^{-d/p}\cdot h^{L} \left\| f^{(L)}\right\|_{L_{p,-\alpha}(\R)}\int_{\R}\w{\by-\bx}{L+\alpha}|\dual(\by+\bl)|  \d\by\nonumber \\
		&= C_{L,\alpha}\cdot h^{-d/p} \cdot h^{L}\left\| f^{(L)}\right\|_{L_{p,-\alpha}(\R)}\int_{\R}\w{\by-\bl-\bx}{L+\alpha}|\dual(\by)|   \d\by\nonumber\\
		&\leq C_{L,\alpha}\cdot h^{-d/p} \cdot h^{L}\left\| f^{(L)}\right\|_{L_{p,-\alpha}(\R)}\w{\bx+\bl}{L+\alpha}\|\dual\|_{L_{1,L+\alpha}(\R)}\label{eq:bound_L1}.
		\end{align}
	The last estimate is again due to the submultiplicativity of the weight $\w{\cdot}{\alpha}$.
	
	Since $\varphi\in W_{q,L+\alpha}(\R)$, it follows from~\cite[Proposition 6]{NguyenU:2015b} that $\dual$ also belongs to $W_{q,L+\alpha}(\R)$. Since $W_{q,L+\alpha}(\R)\subset W_{1,L+\alpha}(\R)= L_{1,L+\alpha}(\R)$, it must be that $\dual \in L_{1,L+\alpha}(\R)$ and so the right-hand side of~\eqref{eq:bound_L1} is finite.
	Plugging~\eqref{eq:bound_L1} into~\eqref{eq:step1} yields
		\begin{align*}
		\left\|e\right\|_{L_{p,-\alpha}(\R)} 
		&\leq 	C_{L,\alpha}\cdot h^L\cdot\left\| f^{(L)}\right\|_{L_{p,-\alpha}(\R)}\|\dual\|_{L_{1,L+\alpha}(\R)}\left(\int_{[0,1]^d} \left(\sum_{\bl\in\Z}  \w{\bx+\bl}{L+\alpha}|\varphi(\bx+\bl)| \right)^p\d \bx\right)^{1/p}\nonumber\\
		&= 	\underbrace{C_{L,\alpha}\cdot\|\dual\|_{L_{1,L+\alpha}(\R)}\cdot\|\varphi\|_{W_{p,L+\alpha}(\R)}}_{C_{\varphi,L,\alpha}}\cdot  h^L\cdot\left\| f^{(L)}\right\|_{L_{p,-\alpha}(\R)},
		\end{align*}
	which is the desired bound.
\end{proof}

With the above results in hands, we are now ready to prove Theorem~\ref{thm:main}.
\begin{proof}[Proof of Theorem~\ref{thm:main}]	
Without loss of generality, assume that $h\in (0,1)$. Put $g\defeq \Jh f$. By using the triangle inequality and by applying Theorem~\ref{thm:projection}, we have that
\begin{align*}
	\wn{f-\Ph f}&\leq \wn{f-g} + \wn{\Ph f - \Ph g} + \wn{g - \Ph g}\\
			   &\leq (1+C_{\varphi,\alpha})\wn{f-g} + \wn{g - \Ph g}.	
\end{align*}
This bound together with Propositions~\ref{thm:f-Jf} and~\ref{thm:error_Jf} immediately implies~\eqref{eq:bound_error}, completing the proof. 	
\end{proof}


\section{Interpolation Error Bound}\label{sec:interp}
We consider in this section the approximation scheme in which a function is ideally sampled (without a prefilter) and reconstructed using an interpolating kernel.
The interpolation operator associated with kernel $\varphi$ and sampling step $h$ is defined by
\begin{align}\label{eq:interp2}
\Ih: f\mapsto \tilde{f}_{\rm int}=\sum_{\bk\in\Z} f(h\bk)\varphi_{\rm int}\left(\frac{\cdot}{h}-\bk\right),
\end{align}
where the interpolant $\varphi_{\rm int}$ is related to the kernel $\varphi$ by 
\begin{align}\label{eq:phi_int}
\varphi_{\rm int} \defeq \sum_{\bk\in\Z} a[\bk] \varphi(\cdot - \bk),
\end{align}
and where the discrete filter $a$ is given in the Fourier domain by
\begin{align}\label{eq:filter}
\hat{a}(\bomega) \defeq \frac{1}{\sum_{\bk\in\Z}\varphi(\bk)\e^{-\j\ip{\bomega}{\bk}}}.
\end{align}
This filter is to make sure that $f(\h\bk) = f_{\rm int} (\h\bk)$, for all $\bk\in\Z$. We have assumed implicitly in~\eqref{eq:filter} that $\sum_{\bk\in\Z}\varphi(\bk)\e^{-\j\ip{\bomega}{\bk}}$ is nonzero for almost all $\bomega\in\R$. It is noteworthy that, in the absence of a prefilter, the function $f$ to be approximated has to be continuous everywhere for the sampling to make sense.   

Another way to express~\eqref{eq:interp2} is
\begin{align}\label{eq:interp1}
\Ih: f\mapsto \tilde{f}_{\rm int}=\sum_{\bk\in\Z} c[\bk]\varphi\left(\frac{\cdot}{h}-\bk\right),
\end{align}
where $c \defeq (\sigma_{1/h} f)[\cdot] * a$ is the sampled sequence of $f$ discretely filtered by $a$. We write $\I$ for $I_{\varphi,1}$.



The following lemma says that the interpolant $\varphi_{\rm int}$ and the kernel $\varphi$ can be made to lie in the same weighted hybrid-norm space by imposing on $\varphi$ some mild conditions that are satisfied by, for example, B-splines of all orders.
\begin{lem}\label{thm:phi_int}
	 Let $1\leq p\leq \infty$ and $\alpha\geq 0$. Let $\varphi\in W_{p,\alpha}(\R)$ such that $\varphi[\cdot]\in \ell_{1,\alpha}(\Z)$ and $\sum_{\bk\in\Z}\varphi[\bk]\e^{-\j\ip{\bomega}{\bk}}$ is nonzero for almost all $\bomega\in\R$. Then, the corresponding interpolant $\varphi_{\rm int}$ defined in~\eqref{eq:phi_int} also belongs to $W_{p,\alpha}(\R)$.
\end{lem}
\begin{proof} Section~\ref{sec:proofs}
\end{proof}

The next result is the interpolation counterpart of Theorem~\ref{thm:projection} and can be thought of as the scaled version of~\cite[Proposition 9]{NguyenU:2015b}. It asserts that $I_{\varphi,h}$ is a bounded operator from $L^{d/p+\varepsilon}_{p,-\alpha}(\R)$ to $V_{p,-\alpha,h}(\varphi)$ whose norm is bounded as  $h\rightarrow 0$. The underlying condition is that the interpolant $\varphi_{\rm int}$ belongs to the weighted hybrid-norm space $W_{p,\alpha}(\R)$.
\begin{thm}\label{thm:interp}
	Assume that $1\leq p \leq \infty$, $\alpha\geq 0$, and $r>d/p$. Let $\varphi\in W_{p,\alpha}(\R)$ such that $\varphi[\cdot]\in \ell_{1,\alpha}(\Z)$ and $\sum_{\bk\in\Z}\varphi[\bk]\e^{-\j\ip{\bomega}{\bk}}$ is nonzero for almost all $\bomega\in\R$. Then, there exists a constant $C_{\varphi,r,\alpha}$ such that, for all continuous functions $f\in L^{r}_{p,-\alpha}(\R)$ and for all $h\in (0,1)$,
	\begin{align}\label{eq:desired_interp2}
		\wn{\Ih f} \leq C_{\varphi,r,\alpha}\cdot \norm{f}_{ L^{r}_{p,-\alpha}(\R)}.
	\end{align}
\end{thm}
\begin{proof}
	Let $B_{r} \defeq \cF^{-1}\{\w{\cdot}{-r}\}$ be the kernel associated with the Bessel potential of order $r$.
	Recall from~\cite[Proposition 6.1.5]{Grafakos:2008b} that $B_r(\bx) > 0$, for all $\bx\in\R$, and that
	\begin{align}\label{eq:Bessel_decay}
	B_r(\bx) \leq C_r\,\e^{-\frac{\|\bx\|}{2}}, \quad \forall\, \|\bx\|\geq 2.
	\end{align}
	Moreover,  since $r>d/p$, it is also known~\cite[Proposition 7]{NguyenU:2015b} that $B_r\in L_{p',\alpha}(\R)$.
	
	Let us now define the weight $w_{h}(\bx)\defeq \w{h\bx}{\alpha}$. Recall that $w_h$ is submultiplicative with the same constant $C_{\alpha}$ for all $h>0$. Observe from~\eqref{eq:interp2} that $\Ih =\sigma_{h}\I\sigma_{1/h}$. Therefore, by a change of variable, we have 
	\begin{align}\label{eq:a}
		\wn{\Ih f} &= \wn{\sigma_h\I\sigma_{1/h} f} = h^{d/p}\cdot\norm{\I \left(\sigma_{1/h}f\right)}_{L_{p,1/w_h}(\R)}.
	\end{align}
	We now invoke~\cite[Proposition 4]{NguyenU:2015b} to get
	\begin{align}\label{eq:b}
		\norm{\I \left(\sigma_{1/h}f\right)}_{L_{p,1/w_h}(\R)} \leq C_{\alpha}\cdot \norm{\varphi_{\rm int}}_{W_{p,w_h}(\R)}  \cdot \norm{\left(\sigma_{1/h}f\right)[\cdot]}_{\ell_{p,1/w_h}(\Z)}.
	\end{align}
	Note that, for all $h\in (0,1)$, $w_h(\bx)\leq \w{\bx}{\alpha}$, and so, the quantity $\norm{\varphi_{\rm int}}_{W_{p,w_h}(\R)}$ is bounded since
	\begin{align}\label{eq:c}
		\norm{\varphi_{\rm int}}_{W_{p,w_h}(\R)} \leq \norm{\varphi_{\rm int}}_{W_{p,\alpha}(\R)},
	\end{align}
	which is finite due to Lemma~\ref{thm:phi_int}. On the other hand, since $f=B_r * \fr$, we can write 
	$$
	\sigma_{1/h}f = h^d\cdot                                                                                                                                                                                                                                                                                                                                                                                                                                                                                                                                                                                                                                                                                                                                                                                                                                                                                                                                                                                                                                                                                                                                                                                                                                                                                                                                                                                                                                                                                                                                                                                                                                                                                                                                                                                                                                                                                                                                                                                                                                                                                                                                                                                                                                                                                                                                                                                                                                                                                                                                                                                                                                                                                                                                                                                                                                                                                                                                                                                                                                                                                                                                                                                                                                                                                                                                                                                                                                                                                                                                                                                                                                                                                                                                                                                                                                                                                                                                                                                                                                                                                                                                                                                                                                                                                                                                                                                                                                                                                                                                                                                                                                                                                                                                                                                                                                                                                                                                                                                                                                                                                                                                                                                   \left(\sigma_{1/h}B_r\right)* \left(\sigma_{1/h}\fr\right),
	$$  
	and apply~\cite[Proposition 5]{NguyenU:2015b} to obtain
	\begin{align}
		\norm{\left(\sigma_{1/h}f\right)[\cdot]}_{\ell_{p,1/w_h}(\Z)} &\leq C_{\alpha}\cdot h^d \norm{\sigma_{1/h}B_r}_{W_{p',w_h}(\R)}\cdot \norm{\sigma_{1/h}\fr}_{L_{p,1/w_h}(\R)}\nonumber\\
		&= C_{\alpha}\cdot h^d \norm{\sigma_{1/h}B_r}_{W_{p',w_h}(\R)}\cdot h^{-d/p}\norm{f}_{L^r_{p,-\alpha}(\R)}\label{eq:d},
	\end{align}
	where~\eqref{eq:d} is due to a change of variable and the definition of the Sobolev norm $\norm{\cdot}_{L^r_{p,-\alpha}(\R)}$.
	Combining~\eqref{eq:a}, \eqref{eq:b}, \eqref{eq:c}, and~\eqref{eq:d}, we arrive at
	\begin{align}\label{eq:combine}
	\wn{\Ih f} &\leq C^2_{\alpha}\cdot \norm{\varphi_{\rm int}}_{W_{p,\alpha}(\R)} \cdot h^d \norm{\sigma_{1/h}B_r}_{W_{p',w_h}(\R)}\cdot \norm{f}_{L^r_{p,-\alpha}(\R)}.
	\end{align}
	Hence, the desired bound~\eqref{eq:desired_interp2} will be achieved if 
	\begin{align}\label{eq:claim}
		 \norm{\sigma_{1/h}B_r}_{W_{p',w_h}(\R)} \leq C_{r,\alpha} \cdot h^{-d},\quad \forall h\in (0,1),
	\end{align}
	for some constant $C_{r,\alpha}$. In the rest of the proof, we will show that this claim is true. Let us put $\TT\defeq [0,1]^d$, $\TT_h \defeq [0,h]^d$, and $B_{r,\alpha} \defeq \w{\cdot}{\alpha} B_r$. From the positivity of $B_r$, it is clear that $B_{r,\alpha}(\bx)>0$, $\forall \bx\in\R$. By the definition of the mixed norm, we express
	\begin{align*}
		\norm{\sigma_{1/h}B_r}_{W_{p',w_h}(\R)} &= \norm{\sigma_{1/h} B_{r,\alpha}}_{W_{p'}(\R)}= \norm{\sum_{\bk\in\Z}\left(\sigma_{1/h} B_{r,\alpha}\right)(\cdot+\bk)}_{L_{p'}(\TT)}\\
										&= h^{-d/p'} \norm{\sum_{\bk\in\Z} B_{r,\alpha}(\cdot+h\bk)}_{L_{p'}(\TT_h)}.
	\end{align*}
	Applying Minkowski's inequality, we get
	\begin{align}
	 \norm{\sigma_{1/h}B_r}_{W_{p',w_h}(\R)} &\leq h^{-d/p'}\sum_{\bk\in\cS_h}\norm{ B_{r,\alpha}(\cdot+h\bk)}_{L_{p'}(\TT_h)} + h^{-d/p'}\norm{\sum_{\bk\in\Z\setminus \cS_h} B_{r,\alpha}(\cdot+h\bk)}_{L_{p'}(\TT_h)}\nonumber\\
	&=: A+B,\label{eq:AB}
	\end{align}
	where $\cS_h$ is a subset of $\Z$ defined by
	\begin{align*}
		\cS_h \defeq \left\{\bk\in\Z: \norm{\bk}\leq \frac{\sqrt{d}+2}{h}\right\}.
	\end{align*}
	We complete the proof by showing that both terms $A$ and $B$ in~\eqref{eq:AB} are bounded by $C_{r,\alpha}\, h^{-d}$. It is clear that $|\cS_h|=C h^{-d}$, for some constant $C$. Therefore, by \Hd\ inequality
	\begin{align}
		A &\leq h^{-d/p'}\cdot\left(\sum_{\bk\in\cS_h}1^p\right)^{1/p} \cdot \left(\sum_{\bk\in\cS_h } \norm{ B_{r,\alpha}(\cdot+h\bk)}_{L_{p'}(\TT_h)}^{p'}\right)^{1/p'}\nonumber\\
		&= h^{-d/p'}\cdot |\cS_h|^{1/p} \cdot \left(\sum_{\bk\in\cS_h } \int_{\TT_h} \left| B_{r,\alpha}(\bx+h\bk)\right|^{p'}\d \bx\right)^{1/p'}\nonumber\\
		&\leq  h^{-d/p'}\cdot C\cdot h^{-d/p} \cdot \left( \sum_{\bk\in\Z }\int_{\TT_h} \left| B_{r,\alpha}(\bx+h\bk)\right|^{p'}\d \bx\right)^{1/p'}\nonumber\\
		&\leq  C \cdot h^{-d} \cdot \left( \int_{\R} \left| B_{r,\alpha}(\bx)\right|^{p'}\d \bx\right)^{1/p'}\nonumber\\
		&\leq \underbrace{ C \cdot \norm{B_{r}}_{L_{p',\alpha}(\R)}}_{C_{r,\alpha}}\cdot h^{-d}.\label{eq:constant}
	\end{align}
The constant $C_{r,\alpha}$ in~\eqref{eq:constant} is finite because $B_r \in L_{p',\alpha}(\R)$. We now proceed to bound the term $B$ in~\eqref{eq:AB}. As $h\in (0,1)$, we have that, for all $\bx\in\TT$ and for all $\bk\notin\cS_h$, 
\begin{align*}
	\norm{h\bx + h\bk} \geq h\norm{\bk} - h\norm{\bx}> (\sqrt{d} + 2 )- h\sqrt{d}> 2,
\end{align*}
which, according to~\eqref{eq:Bessel_decay}, implies that 
$$
	B_{r}(h\bx+h\bk)\leq C_r\, \e^{-\norm{h\bx+h\bk}/2}\leq C_r\, \e^{\frac{\norm{h\bx}-\norm{h\bk}}{2}}.
$$
Plugging this bound into the formula of $B$ and using the submultiplicativity of the weight $\w{\cdot}{\alpha}$ and the fact that $h\in (0,1)$, we get	
\begin{align}
B&= \left(\int_{\TT}\left(\sum_{\bk\in\Z\setminus \cS_h} \w{h\bx+h\bk}{\alpha}B_{r}(h\bx+h\bk)\right)^{p'}\d\bx\right)^{1/p'}\nonumber\\
&\leq C_{r,\alpha}\cdot \left(\int_{\TT}\left(\sum_{\bk\in\Z\setminus \cS_h} \w{\bx}{\alpha}\w{h\bk}{\alpha} \e^{\frac{\norm{\bx}-\norm{h\bk}}{2}}\right)^{p'}\d\bx\right)^{1/p'}\nonumber\\
&\leq C_{r,\alpha}\cdot \left(\int_{\TT}\w{\bx}{p'\alpha}\e^{\frac{p'\norm{\bx}}{2}}\d\bx\right)^{1/p'}\cdot\sum_{\bk\in\Z}  \w{h\bk}{\alpha}\e^{-\frac{\norm{h\bk}}{2}}\cdot \label{eq:sum_int}
\end{align}
Since the integral in~\eqref{eq:sum_int} is a constant independent of $h$, we only need to show that the sum is bounded by $C_\alpha\,h^{-d}$. Again, by the submultiplicativity of the weight $\w{\cdot}{\alpha}$ and by the assumption that $h\in (0,1)$, we have 
\begin{align*}
	\int_{\R} \w{\bx}{\alpha}\e^{-\frac{\|\bx\|}{2}}\d\bx &= \sum_{\bk\in\Z} \int_{\TT_h} \w{\bx+h\bk}{\alpha}\e^{-\frac{\|\bx+h\bk\|}{2}}\d\bx\\
	&\geq C_{\alpha}\int_{\TT_h}\w{\bx}{-\alpha}\e^{-\frac{\|\bx\|}{2}} \d\bx\sum_{\bk\in\Z} \w{h\bk}{\alpha}\e^{-\frac{\|h\bk\|}{2}} \\
	&= C_{\alpha}\cdot h^{d}\int_{\TT}\w{h\bx}{-\alpha}\e^{-\frac{\|h\bx\|}{2}} \d\bx \,\sum_{\bk\in\Z} \w{h\bk}{\alpha}\e^{-\frac{\|h\bk\|}{2}}\\
	&\geq C_{\alpha}\cdot h^{d}\int_{\TT}\w{\bx}{-\alpha}\e^{-\frac{\|\bx\|}{2}} \d\bx \,\sum_{\bk\in\Z} \w{h\bk}{\alpha}\e^{-\frac{\|h\bk\|}{2}},
\end{align*}  
which implies
\begin{align*}
\sum_{\bk\in\Z} \w{h\bk}{\alpha}\e^{-\frac{\|h\bk\|}{2}} &\leq C_{\alpha}^{-1} \int_{\R} \w{\bx}{\alpha}\e^{-\frac{\|\bx\|}{2}}\d\bx \left(\int_{\TT}\w{\bx}{-\alpha}\e^{-\frac{\|\bx\|}{2}}\d\bx\right)^{-1} \cdot h^{-d}\\
&=C_{\alpha}\cdot h^{-d}.
\end{align*}
Combining this with~\eqref{eq:sum_int} yields that $B\leq C_{r,\alpha}\, h^{-d}$ which, together with~\eqref{eq:constant}, establishes the claim~\eqref{eq:claim} and therefore completes the proof. 
\end{proof}

In the rest of this section, we state and prove the interpolation counterpart of Theorem~\ref{thm:main}.
\begin{thm}\label{thm:interp_approx}
	Assume that $1 \leq p \leq \infty$, $L\in\Zplus$, $\alpha\geq 0$, and $r>\Frac{d}{p}$. Let $\varphi$ be an element of $W_{p,L+\alpha}(\R)$   that satisfies the Strang-Fix conditions of order $L$. Assume also that $\varphi[\cdot]\in \ell_{1,L+\alpha}(\Z)$ and $\sum_{\bk\in\Z}\varphi[\bk]\e^{-\j\ip{\bomega}{\bk}}$ is nonzero for almost all $\bomega\in\R$. Then, there exists a constant $C_{\varphi,L,\alpha}$ such that, for all continuous functions $f$ in $H^{L,r}_{p,-\alpha}(\R)$,
	\begin{align} \label{eq:desired_interp}
	\wn{f-\Ih f} \leq C_{\varphi,L,\alpha} \cdot h^{L}\cdot \left\|(\fr)^{(L)}\right\|_{L_{p,-\alpha}},
	\end{align}
	when $h\rightarrow 0$.
\end{thm}

Similar to the proof of Theorem~\ref{thm:main}, we divide the proof of Theorem~\ref{thm:interp_approx} into two propositions. 
\begin{prop}\label{thm:f-Jf2}
	For $1\leq p \leq \infty$, $L\in \Zplus$, $\alpha\geq 0$, $r>0$, and $\Jh$ being the smoothing operator defined in~\eqref{eq:smoothing}, there exists a constant $C_{L,\alpha}$ such that, for all $f\in H^{L,r}_{p,-\alpha}(\R)$ and for all $h\in (0,1)$,
	\begin{align*} 
	\norm{f-\Jh f}_{L^{r}_{p,-\alpha}(\R)} \leq C_{L,\alpha} \cdot\h^{L}\cdot\norm{(\fr)^{(L)}}_{L_{p,-\alpha}(\R)}.
	\end{align*}
\end{prop}
\begin{proof}
	Put  $B_{r}\defeq \cF^{-1}\left\{\w{\cdot}{-r}\right\}$. Since $\Jh$ is a convolution operator, we have the expression
	\begin{align*}
	f - \Jh f = B_r*\fr - \Jh( B_r* \fr) =  B_r*(\fr - \Jh \fr).
	\end{align*} 
	Hence 
	\begin{align}\label{eq:convolution}
		\norm{f-\Jh f}_{L^{r}_{p,-\alpha}(\R)} = \wn{\fr-\Jh \fr}.
	\end{align}
	We now apply Proposition~\ref{thm:f-Jf} to $\fr\in H^{L}_{p,-\alpha}(\R)$ to obtain
	\begin{align}\label{eq:L}
		\wn{\fr-\Jh \fr} \leq C_{L,\alpha} \cdot\h^{L}\cdot\wn{(\fr)^{(L)}}.
	\end{align}
	Putting~\eqref{eq:convolution} and \eqref{eq:L} together completes the proof.

	
\end{proof}

\begin{prop}\label{thm:error_Jf2}
	Assume that $1\leq p \leq \infty$, $L\in\Zplus$, $\alpha\geq 0$, and $r>0$. Let  $\Jh$ be the smoothing operator defined in~\eqref{eq:smoothing}. If $\varphi$ satisfies the conditions of Theorem~\ref{thm:interp_approx}, there exists a constant $C_{\varphi,r,L,\alpha}$ such that, for all $f\in H^{L,r}_{p,-\alpha}(\R)$ and for all $h\in (0,1)$,
	\begin{align}\label{eq:bound_error_Jf2} 
	\left\|\Jh f-\Ih \Jh f\right\|_{L_{p,-\alpha}(\R)} \leq C_{\varphi,r,L,\alpha}\cdot h^L \cdot \norm{(\fr)^{(L)}}_{L_{p,-\alpha}(\R)}.
	\end{align}
\end{prop}
\begin{proof}
	We first show that $f\in H^{L}_{p,-\alpha}(\R) $. Indeed, since $f=B_r*\fr$, where $B_{r}\defeq \cF^{-1}\left\{\w{\cdot}{-r}\right\}$, we have the estimate
	\begin{align}
		\norm{\partial^{\bl}f}_{L_{p,-\alpha}(\R)} & = \wn{\partial^{\bl}(B_r*\fr)} = \wn{B_r*\partial^{\bl}\fr}\nonumber\\
		& \leq C_\alpha\cdot\norm{B_r}_{L_{1,\alpha}(\R)}\cdot\wn{\partial^{\bl}\fr}\label{eq:Young}\\
		&= C_{r,\alpha} \cdot\wn{\partial^{\bl}\fr},\quad \forall |\bl|\leq L\label{eq:constant1}
	\end{align}
	where~\eqref{eq:Young} is a consequence of weighted Young's inequality. On the other hand, it was shown in~\cite[Proposition 7]{NguyenU:2015b} that $B_r\in L_{1,\alpha}(\R)$, for $r>0$. This means that the constant $C_{r,\alpha}$ in~\eqref{eq:constant1} is finite, which then implies that $f\in H^L_{p,-\alpha}(\R)$.
	
	Let $R_{\bx}$ be the remainder of the order-($L-1$) Taylor series of the infinitely differentiable function $g\defeq\Jh\f$  about $\bx$. Since $\varphi_{\rm int}$ is a quasi-interpolant of order $L$, $\Ih$ maps every polynomial of order less than $L$ to itself. Following the path of the proof of Proposition~\ref{thm:error_Jf}, we write
	\begin{align*}
		e(\bx)\defeq  g(\bx)-(\Ih g)(\bx) = - \sum_{\bl\in\Z}c_{\bx}[\bl] \varphi_{\rm int}\left(\frac{\bx}{h} - \bl\right),
	\end{align*}
	where the sequence $c_{\bx}$ is redefined as
	\begin{align*}
		c_{\bx}[\bl] \defeq R_{\bx}(\h\bl), \quad \text{for } \bl\in\Z.
	\end{align*}
	Therefore, \eqref{eq:step1} still holds and we only need to estimate $\norm{c_{\bx,\bl}}_{\ell_{p}(\Z)}$, where 
	\begin{align*}
	c_{\bx,\bl}[\bk]\defeq \w{h\bk}{-\alpha}|c_{\h\bx+h\bk}[\bk-\bl]|=\w{h\bk}{-\alpha}|R_{\h\bx+h\bk}(\h\bk-\h\bl)|.
	\end{align*}
	Similarly to~\eqref{eq:swap}, we express
		\begin{align*}
		R_{\h\bx+\h\bk}(\h\bk - \h\bl) 
		&= \int_{0}^{1}\Jh T_{-\bl,\tau} f(\h\bk) \d\tau,
		\end{align*}
		where the operator $T_{\by,\tau}$ is given in~\eqref{eq:operator}. Repeating the manipulations in the proof of Proposition~\ref{thm:error_Jf}, we obtain the counterpart of~\eqref{eq:bound_L1}: 
	\begin{align*}
	\|c_{\bx,\bl}\|_{\ell_{p}(\Z)} 
	&\leq C_{L,\alpha}\cdot h^{-d/p} \left\| f^{(L)}\right\|_{L_{p,-\alpha}(\R)}\w{\bx+\bl}{L+\alpha}.
	\end{align*}
	Substituting this bound into~\eqref{eq:step1}, we end up with
	\begin{align}\label{eq:derivative_L}
	\left\|e\right\|_{L_{p,-\alpha}(\R)} 
	&\leq C_{L,\alpha}\cdot \|\varphi_{\rm int}\|_{W_{p,L+\alpha}(\R)}\cdot h^{L}\cdot\left\| f^{(L)}\right\|_{L_{p,-\alpha}(\R)},
	\end{align}
	where $\|\varphi_{\rm int}\|_{W_{p,L+\alpha}(\R)}$ is a finite constant thanks to Lemma~\ref{thm:phi_int}. Combining~\eqref{eq:derivative_L} and~\eqref{eq:constant1} gives us the desired bound~\eqref{eq:bound_error_Jf2}.
\end{proof}

\begin{proof}[Proof of Theorem~\ref{thm:interp_approx}]
Without loss of generality, assume that $h\in (0,1)$. Let $g\defeq \Jh f$. By the triangle inequality
\begin{align}\label{eq:triangle}
	\wn{f-\Ih f} \leq \wn{f-g} + \wn{\Ih (f -g) } + \wn{g-\Ih g}.
\end{align}
From~Theorem~\ref{thm:interp} and Propositions~\ref{thm:f-Jf2}, the first two terms in the right-hand side of~\eqref{eq:triangle} are bounded as
\begin{align}
	\wn{f-g} + \wn{\Ih (f -g) } &\leq \norm{f-g}_{L^r_{p,-\alpha}(\R)} +C_{\varphi,r,\alpha}\cdot \norm{f-g}_{L^r_{p,-\alpha}(\R)}\nonumber\\
								&\leq C_{\varphi,r,L,\alpha}\cdot h^{L}\cdot \norm{(\fr)^{(L)}}_{L_{p,-\alpha}(\R)},\label{eq:term1_2}
\end{align}
whereas the third term is also bounded, according to Proposition~\ref{thm:error_Jf2}, as
\begin{align}
	\wn{g-\Ih g} \leq C_{\varphi,r,L,\alpha}\cdot h^{L}\cdot \norm{(\fr)^{(L)}}_{L_{p,-\alpha}(\R)}.\label{eq:term3}
\end{align}
Finally, the desired bound~\eqref{eq:desired_interp} is obtained by combining~\eqref{eq:triangle}, \eqref{eq:term1_2} and~\eqref{eq:term3}. 

\end{proof}
 
\section{Proofs of Auxiliary Results}\label{sec:proofs}
\subsection{Proof of Lemma~\ref{thm:Bspline}}
	It is clear that
	\begin{align*}
	\cF\left\{D^{L}_{\bu}f\right\} = \left(\j\ip{\bu}{\cdot}\right)^{L}\hat{f}.
	\end{align*}
	On the other hand, the Fourier transform of the B-spline $\beta^{L-1}$ is given by~\cite{UnserAE:1993a}
	\begin{align*}
	\hat{\beta}^{L-1}(\omega)=\left(\frac{1-\e^{-\j\omega}}{\j\omega}\right)^L.
	\end{align*}
	Therefore, the Fourier transform of the right-hand side (RHS) of~\eqref{eq:desired} is given by
	\begin{align*}
	\cF\{\rm RHS\} &=  \int_{\RR}\cF \left\{D^{L}_{\bu}f(\cdot - t\bu)\right\} \beta^{L-1}(t)\d t\\
	&=  \int_{\RR}\e^{-\j \ip{t\bu}{\cdot}}\cF \left\{D^{L}_{\bu}f\right\} \beta^{L-1}(t)\d t\\
	&=  \left(\j\ip{\bu}{\cdot}\right)^{L}\hat{f} \cdot \int_{\RR}\e^{-\j \ip{\bu}{\cdot}t}  \beta^{L-1}(t)\d t\\
	&=  \left(\j\ip{\bu}{\cdot}\right)^{L}  \hat{\beta}^{L-1}(\ip{\bu}{\cdot}) \hat{f}\\
	&= \left(1-\e^{-\j\ip{\bu}{\cdot}}\right)^L \hat{f},
	\end{align*}
	which is exactly the Fourier transform of the left-hand side of~\eqref{eq:desired}, completing the proof.
\subsection{Proof of Lemma~\ref{thm:direction}}
	The claim is trivial for $L=0$. We now show~\eqref{eq:direction} based on the induction hypothesis that
	\begin{align}\label{eq:induction}
	\left|D^{L-1}_{\bu}f(\bx)\right| \leq \|\bu\|^{L-1}_{\infty} \cdot f^{(L-1)}(\bx),\quad\forall \bx\in\R.
	\end{align}
	By definition of directional derivatives, we have that  
	\begin{align*}
	\left|D^L_{\bu}f(\bx)\right|&=\left|\sum_{i=1}^{d} u_i\,\frac{\partial}{\partial x_i} D^{L-1}_{\bu}f(\bx)\right|\leq \|\bu\|_{\infty}\cdot\sum_{i=1}^{d}\left|D^{L-1}_{\bu}\frac{\partial f}{\partial x_i} (\bx)\right|.
	\end{align*}
	It then follows from~\eqref{eq:induction} that  
	\begin{align*}
	\left|D^L_{\bu}f(\bx)\right|
	&\leq \|\bu\|_{\infty}\cdot  \|\bu\|^{L-1}_{\infty}\cdot\sum_{i=1}^{d} \left(\frac{\partial f}{\partial x_i}\right)^{(L-1)} (\bx)\\
	&\leq  \|\bu\|^{L}_{\infty}\cdot\sum_{i=1}^{d}\sum_{|\bk|=L-1}\left| \partial^{\bk}\left(\frac{\partial f}{\partial x_i}\right) (\bx)\right|\\
	&= \|\bu\|^L_{\infty} \cdot f^{(L)}(\bx),
	\end{align*}
	completing the proof.
\subsection{Proof of Lemma~\ref{thm:J_bounded}}
	It is clear from the definition of $\Jh$ that $\sigma_{1/h}\Jh = \J \sigma_{1/h}$. Then, we write
	\begin{align}\label{eq:convol}
	\sigma_{1/h}\Jh f = J\sigma_{1/h} f = (\sigma_{1/h}f)* \psi, 
	\end{align}
	where the kernel $\psi$ is given by
	\begin{align*}
	\psi \defeq \sum_{n=1}^{\L} (-1)^{n-1}\binom{\L}{n}\frac{\sigma_{n}\chi}{n^d}.
	\end{align*} 
	Since $\chi$ is a compactly supported smooth function, it is easy to see that the kernel $\psi$ given above is an element of the hybrid-norm space $W_{\infty,\alpha}(\R)$, which is clearly a subspace of $W_{p',\alpha}(\R)$. Then, the convolution expression in~\eqref{eq:convol} allows us to invoke~\cite[Proposition 5]{NguyenU:2015b} to obtain
	\begin{align}
	\left\|(\sigma_{\Frac{1}{h}}\Jh f)[\cdot]\right\|_{\ell_{p,1/w_h}(\Z)} &\leq C_\alpha\, \norm{\psi}_{W_{p',w_h}} \cdot\left\|\sigma_{1/h}f\right\|_{L_{p,1/w_h}(\R)}\nonumber\\
	&\leq C_\alpha\, \norm{\psi}_{W_{p',\alpha}} \cdot\left\|\sigma_{1/h}f\right\|_{L_{p,1/w_h}(\R)}\label{eq:bounded_h}\\
	&= C_\alpha\,  \norm{\psi}_{W_{p',\alpha}} \cdot h^{-d/p} \cdot \left\|f\right\|_{L_{p,-\alpha}(\R)}\label{eq:changeVariable},
	\end{align}
	where~\eqref{eq:bounded_h} is due to the assumption that $h \in (0,1)$ and~\eqref{eq:changeVariable} is the result of a change of variable. Putting $C_{L,\alpha}\defeq  C_\alpha \norm{\psi}_{W_{p',\alpha}}$ gives us the desired bound~\eqref{eq:bound_Jf}.
\subsection{Proof of Lemma~\ref{thm:phi_int}}
	Recall that, for $\alpha \geq 0$, the weight $\w{\cdot}{\alpha}$ is submultiplicative and satisfies the Gelfand-Raikov-Shilov condition. Since $\varphi[\cdot]\in \ell_{1,\alpha}(\Z)$ and since $\sum_{\bk\in\Z}\varphi[\bk]\e^{-\j\ip{\bomega}{\bk}}$ is nonzero for almost all $\bomega\in\R$, we are allowed to invoke the weighted version of Wiener's lemma~\cite[Theorem 6.2]{Grochenig:2006} to deduce that the sequence $a$ defined in~\eqref{eq:filter} also belongs to $\ell_{1,\alpha}(\Z)$. Now that $\varphi_{\rm int}$ has the representation~\eqref{eq:phi_int} with $a\in \ell_{1,\alpha}(\Z)$ and $\varphi\in W_{p,\alpha}(\R)$, it must be that $\varphi_{\rm int}\in W_{p,\alpha}(\R)$ as a consequence of~\cite[Lemma 1]{NguyenU:2015b}.
\bibliographystyle{IEEEtran}
\bibliography{hqn_lib}

\end{document}